\documentclass[preprint,12pt]{elsarticle}
\usepackage{amssymb}
\usepackage{amsthm}
\usepackage{amsmath}
\usepackage{amsfonts}
\usepackage{xcolor}

\newtheorem{theorem}{Theorem}

\newtheorem{corollary}{Corollary}

\newtheorem{conclusion}{Conclusion}
\newtheorem{definition}{Definition}
\usepackage{ragged2e}
\journal{arXiv}

\begin{document}

\begin{frontmatter}



\title{Closed-form formula of Riemann zeta function and eta function for all non-zero given complex numbers via sums of powers of complex functions to disprove Riemann hypothesis}


\author{Dagnachew Jenber Negash}

\address{Addis Ababa Science and Technology University\\Addis Ababa, Ethiopia\\Email: $\text{djdm}\_101979@\text{yahoo.com}$}
\begin{abstract}
An explicit identity of sums of powers of complex functions presented via this a closed-form formula of Riemann zeta function produced at any given non-zero complex numbers. The closed-form formula showed us Riemann zeta function has no unique solution for any given non-zero complex numbers which means Riemann zeta function is entirely divergent. Infinitely many zeros of Riemann zeta function produced unfortunately those zeros also gives us non-zero values of Riemann zeta function. Among those zeros some of them are zeros of Riemann hypothesis. The present paper also discussed on eta function(alternating Riemann zeta function) with exactly the same behavior as Riemann zeta function.
\end{abstract}


\end{frontmatter}

%\linenumbers

\section{Introduction}
\label{S:1}
\justify
In his 1859 paper on the number of primes less than a given magnitude, bernhard Riemann (1826-1866) examined the properties of the function
\begin{equation*}
\zeta(s)=\sum_{n=1}^{\infty}n^{-s}
\end{equation*}
for {\bf s} a complex number. This function is analytic for real part of {\bf s} greater than 1 and is related to the prime numbers by the euler product formula
\begin{equation*}
\zeta(s)=\prod_{p(prime)}(1-p^{-s})^{-1}
\end{equation*}
again defined for real part of {\bf s} greater than one. This function extends to points with real part {\bf s} greater than zero by the formula(among others)
\begin{equation*}
\zeta(s)=(1-2^{(1-s)})^{-1}\sum_{n=1}^{\infty}(-1)^{n-1}n^{-s}
\end{equation*}
\subsection*{\bf The Riemann Hypothesis}
\label{S:12}
\justify
The zeta function has no zeros in the region where the real part of {\bf s} is greater than or equal to one. In the region with real part of {\bf s} less than or equal to zero the zeta function has zeros at the negative even integers; these are known as the trivial zeros. All remaining zeros lie in the strip where the real part of {\bf s} is strictly between {\bf 0} and {\bf 1} ({\bf the critical strip}). It is known that there are infinitely many zeros on the line {\bf 1/2+it} as {\bf t} ranges over the real numbers. This line in the complex plane is known as the {\bf critical line}. The Riemann Hypothesis ({\bf RH}) is that all non-trivial zeros of the zeta function lie on the critical line \cite{[1]},\cite{[6]},\cite{[7]}.\\ Let's say that again:
\justify
{\bf Riemann Hypothesis:}\\
All non-trivial zeros of the zeta function lie on the line {\bf 1/2+it} as {\bf t} ranges over the real numbers.
\subsection*{\bf The Functional Equation}
\label{S:13}
\justify
The functional equation of the zeta function is
\begin{equation*}
\zeta(s)=\Gamma(1-s)(2\pi)^{s-1}2\sin\left( \frac{\pi s}{2}\right)\zeta(1-s)
\end{equation*}
From which values of the zeta function at {\bf s} can be computed from its values at {\bf (1-s)}. Using this equation one sees immediately that the zeta function is zero at the negative even integers\cite{[3]},\cite{[5]},\cite{[11]},\cite{[12]}.
\justify
Many famous mathematicians studied and developed equations to prove Riemann Hypothesis in different approach\cite{[2]},\cite{[8]},\cite{[9]}. An analytic continuation of Riemann zeta function developed which agrees for all complex numbers \cite{[4]}. 
\justify
The present paper gives us closed-form formula of Riemann zeta function which spits out non-unique solutions for each complex numbers excepting zero.
\section{Four Definitions}
\label{S:21}
Throughout this paper we will use the following definitions.
\begin{definition}
For every $k\in \mathbb{Z^{+}}|_{\{1\}}$ , $a_1,a_2\in \mathbb{R}$ and $d_1,d_2 \in \mathbb{R}|_{\{0\}}$, define the following
\begin{equation}
B^{1}:=B^{1}(k,a_1,a_2,d_1,d_2)=\frac{1}{\ln(k)}\arctan\bigg(\frac{a_2+(k-1)d_2}{a_1+(k-1)d_1}\bigg)
\end{equation}

\begin{equation}
A^{1}:=A^{1}(k,a_1,d_1,B^{1})=\frac{1}{\ln(k)}\ln\bigg(\frac{a_1+(k-1)d_1}{\cos(B^{1}\ln(k))}\bigg)
\end{equation}
\begin{equation}
A^{2}:=A^{2}(k,a_2,d_2,B^{1})=\frac{1}{\ln(k)}\ln\bigg(\frac{a_2+(k-1)d_2}{\sin(B^{1}\ln(k))}\bigg)
\end{equation}
\end{definition}
\begin{definition}
Define the Riemann zeta function, $\zeta(s)$, for all complex numbers, $s$ with real part greater than one by:
\begin{equation*}
\zeta(s)=\sum_{n=1}^{\infty}n^{-s}
\end{equation*}
\end{definition}
\begin{definition}
Define the eta function(alternating Riemann zeta function), $\eta(s)$, for all complex numbers, $s$ with real part greater than zero by:
\begin{equation*}
\eta(s)=\sum_{n=1}^{\infty}(-1)^{n-1}n^{-s}
\end{equation*}
\end{definition}
\begin{definition}
Let {\bf s} be a complex number with real part of {\bf s} greater than zero. Then we define the gamma function, $\Gamma(s)$ as
\begin{equation*}
\Gamma(s)=\int_{0}^{\infty}e^{-t}t^{s-1}dt
\end{equation*}
Note that $\Gamma(s+1)=s\Gamma(s)$ with real part of a complex number {\bf s} greater than zero.
\end{definition}

\section{Two Theorems}
\label{S:2}
\begin{theorem}\cite{[10]} For every complex numbers $a$ and $d$, where $d\ne 0$, then we have
\begin{equation}
\frac{d}{(n-1)!(n-3)!}\sum_{r=1}^{k}(a+(r-1)d)^{n-1}=\sum_{i=0}^{n-3}\frac{1}{i!(n-i)!(n-3-i)!} \bigg(\frac{d}{2}\bigg)^i(-1)^iS_{n-i}
\end{equation}
Where 
\begin{equation*}
S_{n-i}
\end{equation*}
\begin{equation}
=\bigg(\frac{n-i}{2}-1\bigg)kd^{n-i}-\frac{n-i}{2}d^{n-i-2}((a+kd)^2-a^2)+(a+kd)^{n-i}-a^{n-i}
\end{equation}
\end{theorem}

\begin{theorem}\cite{[10]}For every complex numbers $a$ and $d$, where $d\ne 0$, then we have
\begin{equation}
\sum_{r=1}^{k}(-1)^{(r-1)}(a+(r-1)d)^{n-1}=\frac{1}{nd}\bigg[\sum_{i=0}^{n-3}\binom {n-3} i \bigg(\frac{d}{2}\bigg)^i\frac{n!}{(n-i)!}(-1)^{(i+1)}L_{n-i}\bigg]
\end{equation}
Where 
\begin{equation*}
L_{n-i}
\end{equation*}
\begin{equation}
=\bigg(\frac{n-i}{2}-1\bigg)kd^{n-i}+\frac{n-i}{2}d^{n-i-2}((a+kd-d)^2-(a-d)^2)+(-1)^{(n-i-1)}[(a+kd-d)^{n-i}-(a-d)^{n-i}]
\end{equation}
\end{theorem}
\section{Main Result}
\label{S:3}
%\subsection{\bf Riemann zeta function, $\zeta(Z)$ at complex functions, Z}
%\label{S:31}
\begin{theorem}Define a complex numbers $a=a_1+a_2i$, $d=d_1+d_2i\ne(0,0)$, $s=A^{1}+iB^{1}$ or $s=A^{2}+iB^{1}$ and $1\ne x\in \mathbb{R^+}$. If 
\begin{equation*}
a+(x-1)d=x^s
\end{equation*}then
\begin{equation*}
B^1:=\frac{1}{\ln(x)}\arctan\bigg(\frac{a_2+(x-1)d_2}{a_1+(x-1)d_1}\bigg)
\end{equation*}
 
\begin{equation*}
A^1:=\frac{1}{\ln(x)}\ln\bigg(\frac{a_1+(x-1)d_1}{\cos(B^1\ln(x))}\bigg)
\end{equation*}

\begin{equation*}
A^2:=\frac{1}{\ln(x)}\ln\bigg(\frac{a_2+(x-1)d_2}{\sin(B^1\ln(x))}\bigg)
\end{equation*}
\end{theorem}

\begin{proof}

\begin{equation*}
x^s=a+(x-1)d
\end{equation*}

\begin{equation*}
\Rightarrow e^{\ln(x^s)}=a+(x-1)d
\end{equation*}

\begin{equation*}
\Rightarrow e^{s\ln(x)}=a+(x-1)d
\end{equation*}

\begin{equation*}
\Rightarrow e^{(A+Bi)\ln(x)}=a+(x-1)d
\end{equation*}

\begin{equation*}
\Rightarrow e^{(A\ln(x)+iB\ln(x))}=a+(x-1)d
\end{equation*}

\begin{equation*}
\Rightarrow e^{A\ln(x)}e^{iB\ln(x)}=a+(x-1)d
\end{equation*}

\begin{equation*}
\Rightarrow e^{A\ln(x)}\bigg(\cos(B\ln(x))+i\sin(B\ln(x))\bigg)=a+(x-1)d
\end{equation*}

\begin{equation*}
\Rightarrow e^{A\ln(x)}\cos(B\ln(x))+ie^{A\ln(x)}\sin(B\ln(x))=a+(x-1)d
\end{equation*}

\begin{equation*}
\Rightarrow e^{A\ln(x)}\cos(B\ln(x))+ie^{A\ln(x)}\sin(B\ln(x))=a_1+a_2i+(x-1)(d_1+d_2i)
\end{equation*}

\begin{equation*}
\Rightarrow e^{A\ln(x)}\cos(B\ln(x))+ie^{A\ln(x)}\sin(B\ln(x))=a_1+(x-1)d_1+(a_2+(x-1)d_2)i
\end{equation*}

\begin{equation}
\Rightarrow e^{A\ln(x)}\cos(B\ln(x))=a_1+(x-1)d_1
\end{equation}

\begin{equation}
\Rightarrow e^{A\ln(x)}\sin(B\ln(x))=a_2+(x-1)d_2
\end{equation}
Now divide equation $\color{blue} (9)$ by equation $\color{blue} (8)$, then we have

\begin{equation*}
\tan(B\ln(x))=\frac{\sin(B\ln(x))}{\cos(B\ln(x))}=\frac{a_2+(x-1)d_2}{a_1+(x-1)d_1}
\end{equation*}

\begin{equation*}
\Rightarrow B=\frac{1}{\ln(x)}\arctan\bigg(\frac{a_2+(x-1)d_2}{a_1+(x-1)d_1}\bigg)
\end{equation*}
We can let that
\begin{equation*}
B^{1}=B^{1}(x,a_1,a_2,d_1,d_2)=B=\frac{1}{\ln(x)}\arctan\bigg(\frac{a_2+(x-1)d_2}{a_1+(x-1)d_1}\bigg)
\end{equation*}

\begin{equation*}
\Rightarrow B^{1}=B^{1}(x,a_1,a_2,d_1,d_2)=\frac{1}{\ln(x)}\arctan\bigg(\frac{a_2+(x-1)d_2}{a_1+(x-1)d_1}\bigg)
\end{equation*}

From equation $\color{blue}(8)$, we have
\begin{equation*}
e^{A\ln(x)}\cos(B^1\ln(x))=a_1+(x-1)d_1
\end{equation*}

\begin{equation*}
\Rightarrow e^{A\ln(x)}=\frac{a_1+(x-1)d_1}{\cos(B^1\ln(x))}
\end{equation*}

\begin{equation*}
\Rightarrow A\ln(x)=\ln\bigg(e^{A\ln(x)}\bigg)=\ln\bigg(\frac{a_1+(x-1)d_1}{\cos(B^1\ln(x))}\bigg)
\end{equation*}

\begin{equation*}
\Rightarrow A=\frac{1}{\ln(x)}\ln\bigg(\frac{a_1+(x-1)d_1}{\cos(B^1\ln(x))}\bigg)
\end{equation*}
We can let that

\begin{equation*}
 A^{1}:=A^{1}(x,a_1,d_1,B^{1})=A=\frac{1}{\ln(x)}\ln\bigg(\frac{a_1+(x-1)d_1}{\cos(B^1\ln(x))}\bigg)
\end{equation*}

\begin{equation*}
 \Rightarrow A^{1}:=A^{1}(x,a_1,d_1,B^{1})=\frac{1}{\ln(x)}\ln\bigg(\frac{a_1+(x-1)d_1}{\cos(B^1\ln(x))}\bigg)
\end{equation*}

And now from equation $\color{blue}(9)$, we have
\begin{equation*}
e^{A\ln(x)}\sin(B^1\ln(x))=a_2+(x-1)d_2
\end{equation*}

\begin{equation*}
\Rightarrow A=\frac{1}{\ln(x)}\ln\bigg(\frac{a_2+(x-1)d_2}{\sin(B^1\ln(x))}\bigg)
\end{equation*}
We can let that
\begin{equation*}
 A^{2}:=A^{2}(x,a_2,d_2,B^1)=A=\frac{1}{\ln(x)}\ln\bigg(\frac{a_2+(x-1)d_2}{\sin(B^1\ln(x))}\bigg)
\end{equation*}

\begin{equation*}
\Rightarrow A^{2}:=A^{2}(x,a_2,d_2,B^1)=\frac{1}{\ln(x)}\ln\bigg(\frac{a_2+(x-1)d_2}{\sin(B^1\ln(x))}\bigg)
\end{equation*}

\end{proof}
%\begin{corollary}If we define a complex numbers $a=a_1+a_2i$, $d=d_1+d_2i\ne(0,0)$, $s=A^{1}+iB^{1}$ or $s=A^{2}+iB^{1}$ and $n,j\in \mathbb{Z^+}$. Then 
%\begin{equation*}
%\sum_{x=1}^{n}(a+(x-1)d)^{j}=\sum_{x=1}^{n}(x^s)^j
%\end{equation*}where
%\begin{equation*}
%B^1:=\frac{1}{\ln(x)}\arctan\bigg(\frac{a_2+(x-1)d_2}{a_1+(x-1)d_1}\bigg)
%\end{equation*}
 
%\begin{equation*}
%A^1:=\frac{1}{\ln(x)}\ln\bigg(\frac{a_1+(x-1)d_1}{\cos(B^1\ln(x))}\bigg)
%\end{equation*}

%\begin{equation*}
%A^2:=\frac{1}{\ln(x)}\ln\bigg(\frac{a_2+(x-1)d_2}{\sin(B^1\ln(x))}\bigg)
%\end{equation*}
%\end{corollary}
\begin{theorem} If $s=A^1+B^1i$ or $s=A^2+B^1i$, then for every $1\ne x\in \mathbb{R}^{+}$
\begin{equation*}
x^s
\end{equation*}

\begin{equation*}
=\frac{a_1+(x-1)d_1}{\cos(B^1\ln(x))}\bigg[\cos\bigg(\arctan\bigg(\frac{a_2+(x-1)d_2}{a_1+(x-1)d_1}\bigg)\bigg)
\end{equation*}

\begin{equation*}
+i\sin\bigg(\arctan\bigg(\frac{a_2+(x-1)d_2}{a_1+(x-1)d_1}\bigg)\bigg)\bigg]
\end{equation*}

\begin{equation*}
x^s
\end{equation*}

or 
\begin{equation*}
=\frac{a_2+(x-1)d_2}{\sin(B^1\ln(x))}\bigg[\cos\bigg(\arctan\bigg(\frac{a_2+(x-1)d_2}{a_1+(x-1)d_1}\bigg)\bigg)
\end{equation*}

\begin{equation*}
+i\sin\bigg(\arctan\bigg(\frac{a_2+(x-1)d_2}{a_1+(x-1)d_1}\bigg)\bigg)\bigg]
\end{equation*}
\end{theorem}

\begin{proof}
\begin{equation*}
x^s=e^{\ln(x^s)}=e^{s\ln(x)}=e^{(A^1+B^1i)\ln(x)}=e^{A^1\ln(x)}e^{iB^1\ln(x)}
\end{equation*}

\begin{equation*}
\Rightarrow x^s=e^{A^1\ln(x)}e^{iB^1\ln(x)}=e^{\bigg(\ln\bigg(\frac{a_1+(x-1)d_1}{\cos(B^1\ln(x))}\bigg)\bigg)}e^{iB^1\ln(x)}
\end{equation*}

\begin{equation*}
\Rightarrow x^s=\frac{a_1+(x-1)d_1}{\cos(B^{1}\ln(x))}e^{iB^{1}\ln(x)}
\end{equation*}

\begin{equation*}
\Rightarrow x^s=\frac{a_1+(x-1)d_1}{\cos(B^{1}\ln(x))}\bigg(\cos(B^{1}\ln(x))+i\sin(B^{1}\ln(x))\bigg)
\end{equation*}
Hence
\begin{equation*}
x^s
\end{equation*}

\begin{equation*}
=\frac{a_1+(x-1)d_1}{\cos(B^{1}\ln(x))}\bigg[\cos\bigg(\arctan\bigg(\frac{a_2+(x-1)d_2}{a_1+(x-1)d_1}\bigg)\bigg)
\end{equation*}

\begin{equation}
+i\sin\bigg(\arctan\bigg(\frac{a_2+(x-1)d_2}{a_1+(x-1)d_1}\bigg)\bigg)\bigg]
\end{equation}

And

\begin{equation*}
x^s=e^{A^{2}\ln(x)}e^{iB^{1}\ln(x)}=e^{\bigg(\ln\bigg(\frac{a_2+(x-1)d_2}{\sin(B^{1}\ln(x))}\bigg)\bigg)}e^{iB^{1}\ln(x)}
\end{equation*}
Hence

\begin{equation*}
x^s
\end{equation*}

\begin{equation*}
=\frac{a_2+(x-1)d_2}{\sin(B^{1}\ln(x))}\bigg[\cos\bigg(\arctan\bigg(\frac{a_2+(x-1)d_2}{a_1+(x-1)d_1}\bigg)\bigg)
\end{equation*}

\begin{equation}
+i\sin\bigg(\arctan\bigg(\frac{a_2+(x-1)d_2}{a_1+(x-1)d_1}\bigg)\bigg)\bigg]
\end{equation}

\end{proof}

\begin{theorem} If $s=A^1+B^1i$ or $s=A^2+B^1i$, then for every $x\in \mathbb{R}^{+}$
\begin{equation*}
x^{-s}
\end{equation*}

\begin{equation*}
=\frac{\cos(B^{1}\ln(x))}{a_1+(x-1)d_1}\bigg[\cos\bigg(\arctan\bigg(\frac{a_2+(x-1)d_2}{a_1+(x-1)d_1}\bigg)\bigg)
\end{equation*}

\begin{equation*}
-i\sin\bigg(\arctan\bigg(\frac{a_2+(x-1)d_2}{a_1+(x-1)d_1}\bigg)\bigg)\bigg]
\end{equation*}
Or
\begin{equation*}
x^{-s}
\end{equation*}

\begin{equation*}
=-\frac{\sin(B^{1}\ln(x))}{a_2+(x-1)d_2}\bigg[\cos\bigg(\arctan\bigg(\frac{a_2+(x-1)d_2}{a_1+(x-1)d_1}\bigg)\bigg)
\end{equation*}

\begin{equation*}
-i\sin\bigg(\arctan\bigg(\frac{a_2+(x-1)d_2}{a_1+(x-1)d_1}\bigg)\bigg)\bigg]
\end{equation*}
\end{theorem}

\begin{proof}
\begin{equation*}
x^{-s}=e^{\ln(x^{-s})}=e^{-s\ln(x)}=e^{(-A^1-B^1i)\ln(x)}=e^{-A^1\ln(x)}e^{-iB^1\ln(x)}
\end{equation*}

\begin{equation*}
\Rightarrow x^{-s}=e^{-A^1\ln(x)}e^{-iB^1\ln(x)}=e^{\bigg(-\ln\bigg(\frac{a_1+(x-1)d_1}{\cos(B^1\ln(x))}\bigg)\bigg)}e^{-iB^1\ln(x)}
\end{equation*}

\begin{equation*}
\Rightarrow x^{-s}=\bigg[\frac{a_1+(x-1)d_1}{\cos(B^{1}\ln(x))}\bigg]^{-1}e^{-iB^{1}\ln(x)}
\end{equation*}

\begin{equation*}
\Rightarrow x^{-s}=\frac{\cos(B^{1}\ln(x))}{a_1+(x-1)d_1}\bigg(\cos(B^{1}\ln(x))-i\sin(B^{1}\ln(x))\bigg)
\end{equation*}
Therefore
\begin{equation*}
x^{-s}
\end{equation*}

\begin{equation*}
=\frac{\cos(B^{1}\ln(x))}{a_1+(x-1)d_1}\bigg[\cos\bigg(\arctan\bigg(\frac{a_2+(x-1)d_2}{a_1+(x-1)d_1}\bigg)\bigg)
\end{equation*}

\begin{equation}
-i\sin\bigg(\arctan\bigg(\frac{a_2+(x-1)d_2}{a_1+(x-1)d_1}\bigg)\bigg)\bigg]
\end{equation}

And 

\begin{equation*}
x^{-s}=e^{-A^{2}\ln(x)}e^{-iB^{1}\ln(x)}=e^{\bigg(-\ln\bigg(\frac{a_2+(x-1)d_2}{\sin(B^{1}\ln(x))}\bigg)\bigg)}e^{-iB^{1}\ln(x)}
\end{equation*}
Hence

\begin{equation*}
x^{-s}
\end{equation*}

\begin{equation*}
=-\frac{\sin(B^{1}\ln(x))}{a_2+(x-1)d_2}\bigg[\cos\bigg(\arctan\bigg(\frac{a_2+(x-1)d_2}{a_1+(x-1)d_1}\bigg)\bigg)
\end{equation*}

\begin{equation}
-i\sin\bigg(\arctan\bigg(\frac{a_2+(x-1)d_2}{a_1+(x-1)d_1}\bigg)\bigg)\bigg]
\end{equation}
\end{proof}

\begin{theorem}If $s=A^1+B^1i$ or $s=A^2+B^1i$ and for every $x\in \mathbb{R}^{+}$, then
\begin{equation*}
x^{-s}\longrightarrow 0 \text{ as }x\longrightarrow \infty
\end{equation*}
\end{theorem}

\begin{proof}
Since 
\begin{equation*}
B^{1}=B^{1}(x,a_1,a_2,d_1,d_2)=\frac{1}{\ln(x)}\arctan\bigg(\frac{a_2+(x-1)d_2}{a_1+(x-1)d_1}\bigg)
\end{equation*}

\begin{equation*}
\Rightarrow B^{1}\ln(x)=\arctan\bigg(\frac{a_2+(x-1)d_2}{a_1+(x-1)d_1}\bigg)
\end{equation*}
From equation $\color{blue}(12)$, we have
\begin{equation*}
x^{-s}
\end{equation*}

\begin{equation*}
=\frac{\cos(B^{1}\ln(x))}{a_1+(x-1)d_1}\bigg[\cos\bigg(\arctan\bigg(\frac{a_2+(x-1)d_2}{a_1+(x-1)d_1}\bigg)\bigg)
-i\sin\bigg(\arctan\bigg(\frac{a_2+(x-1)d_2}{a_1+(x-1)d_1}\bigg)\bigg)\bigg]
\end{equation*}

\begin{equation*}
\Rightarrow x^{-s}
=\frac{\cos\bigg(\arctan\bigg(\frac{a_2+(x-1)d_2}{a_1+(x-1)d_1}\bigg)\bigg)}{a_1+(x-1)d_1}\left[e^{-i\arctan\bigg(\frac{a_2+(x-1)d_2}{a_1+(x-1)d_1}\bigg)}\right]
\end{equation*}

\begin{equation*}
\Rightarrow x^{-s}
=\frac{1}{a_1+(x-1)d_1}\left(\frac{\cos\bigg(\arctan\bigg(\frac{a_2+(x-1)d_2}{a_1+(x-1)d_1}\bigg)\bigg)}{e^{i\arctan\Bigg(\frac{a_2+(x-1)d_2}{a_1+(x-1)d_1}\bigg)}}\right)
\end{equation*}
Now when we apply limit as $x\longrightarrow \infty $ both sides, we get $x^{-s}\longrightarrow 0$. Similarly from equation $\color{blue}(13)$, we get the same result.
\end{proof}

\begin{theorem}If for all $x\in \mathbb{R}^{+}$ with $x^{s}=a+(x-1)d$ where $s=A^{1}+iB^{1}$ or $s=A^{2}+iB^{1}$ with $A^{1},A^{2}>1$ and $a=a_1+a_2i$, $d=d_1+d_2i \ne(0,0)$ are complex numbers, then 
\begin{equation}
\frac{d}{(n-1)!(n-3)!}\sum_{r=1}^{\infty}(r^{-s})^{n-1}=\sum_{i=0}^{n-3}\frac{1}{i!(n-i)!(n-3-i)!} \bigg(\frac{d}{2}\bigg)^i(-1)^iS_{n-i}
\end{equation}
where 
\begin{equation*}
S_{n-i}
\end{equation*}
\begin{equation}
=\left(\frac{n-i}{2}-1\right)\left(d-a\right)d^{n-i-1}-\frac{n-i}{2}d^{n-i-2}\left(d^2-a^2\right)+d^{n-i}-a^{n-i}
\end{equation}

\end{theorem}

\begin{proof}
For $k \in \mathbb{Z^{+}}$ and $r=1,2,3,\cdots,k$. Define $a+(r-1)d=r^s$.

\begin{equation*}
\Rightarrow r=\bigg( \frac{r^{s}-a}{d}\bigg)+1
\end{equation*}
From Theorem $(1)$, we have
\begin{equation*}
\frac{d}{(n-1)!(n-3)!}\sum_{r=1}^{k}(a+(r-1)d)^{n-1}=\sum_{i=0}^{n-3}\frac{1}{i!(n-i)!(n-3-i)!} \bigg(\frac{d}{2}\bigg)^i(-1)^iS_{n-i}
\end{equation*}
where 
\begin{equation*}
S_{n-i}
\end{equation*}
\begin{equation*}
=\bigg(\frac{n-i}{2}-1\bigg)kd^{n-i}-\frac{n-i}{2}d^{n-i-2}((a+kd)^2-a^2)+(a+kd)^{n-i}-a^{n-i}
\end{equation*}
This implies
\begin{equation*}
\frac{d}{(n-1)!(n-3)!}\sum_{r=1}^{k}(r^s)^{n-1}=\sum_{i=0}^{n-3}\frac{1}{i!(n-i)!(n-3-i)!} \bigg(\frac{d}{2}\bigg)^i(-1)^iS_{n-i}
\end{equation*}
where 
\begin{equation*}
S_{n-i}
\end{equation*}

\begin{equation*}
=\bigg(\frac{n-i}{2}-1\bigg)\bigg( \frac{k^{s}-a}{d}+1\bigg)d^{n-i}-\frac{n-i}{2}d^{n-i-2}\bigg(\bigg(a+\bigg( \frac{k^{s}-a}{d}+1\bigg)d\bigg)^2-a^2\bigg)
\end{equation*}

\begin{equation*}
+\bigg(a+\bigg( \frac{k^{s}-a}{d}+1\bigg)d\bigg)^{n-i}-a^{n-i}
\end{equation*}
Replacing $s$ by $-s$, then we have
\begin{equation*}
\frac{d}{(n-1)!(n-3)!}\sum_{r=1}^{k}(r^{-s})^{n-1}=\sum_{i=0}^{n-3}\frac{1}{i!(n-i)!(n-3-i)!} \bigg(\frac{d}{2}\bigg)^i(-1)^iS_{n-i}
\end{equation*}
where 
\begin{equation*}
S_{n-i}
\end{equation*}

\begin{equation*}
=\bigg(\frac{n-i}{2}-1\bigg)\bigg( \frac{k^{-s}-a}{d}+1\bigg)d^{n-i}-\frac{n-i}{2}d^{n-i-2}\bigg(\bigg(a+\bigg( \frac{k^{-s}-a}{d}+1\bigg)d\bigg)^2-a^2\bigg)
\end{equation*}

\begin{equation*}
+\bigg(a+\bigg( \frac{k^{-s}-a}{d}+1\bigg)d\bigg)^{n-i}-a^{n-i}
\end{equation*}
Now apply limit as $k\longrightarrow \infty$ both sides, then we have
\begin{equation*}
\frac{d}{(n-1)!(n-3)!}\sum_{r=1}^{\infty}(r^{-s})^{n-1}=\sum_{i=0}^{n-3}\frac{1}{i!(n-i)!(n-3-i)!} \bigg(\frac{d}{2}\bigg)^i(-1)^iS_{n-i}
\end{equation*}
where 
\begin{equation*}
S_{n-i}
\end{equation*}
\begin{equation*}
=\bigg(\frac{n-i}{2}-1\bigg)\bigg( \frac{-a}{d}+1\bigg)d^{n-i}-\frac{n-i}{2}d^{n-i-2}\bigg(\bigg(a+\bigg( \frac{-a}{d}+1\bigg)d\bigg)^2-a^2\bigg)
\end{equation*}

\begin{equation*}
+\bigg(a+\bigg( \frac{-a}{d}+1\bigg)d\bigg)^{n-i}-a^{n-i}
\end{equation*}
\begin{equation*}
=\bigg(\frac{n-i}{2}-1\bigg)\bigg(d-a\bigg)d^{n-i-1}-\frac{n-i}{2}d^{n-i-2}\bigg(d^2-a^2\bigg)+d^{n-i}-a^{n-i}
\end{equation*}

\end{proof}
\begin{theorem}If for all $x\in \mathbb{R}^{+}$ with $x^{s}=a+(x-1)d$ where $s=A^{1}+iB^{1}$ or $s=A^{2}+iB^{1}$ with $A^{1},A^{2}>0$ and $a=a_1+a_2i$, $d=d_1+d_2i \ne(0,0)$ are complex numbers, then 
\begin{equation}
\sum_{r=1}^{\infty}(-1)^{(r-1)}(r^{-s})^{n-1}=\frac{1}{nd}\bigg[\sum_{i=0}^{n-3}\binom {n-3} i \bigg(\frac{d}{2}\bigg)^i\frac{n!}{(n-i)!}(-1)^{(i+1)}L_{n-i}\bigg]
\end{equation}
where 
\begin{equation*}
L_{n-i}
\end{equation*}
\begin{equation}
=\bigg(\frac{n-i}{2}-1\bigg)(d-a)d^{n-i-1}-\left(\frac{n-i}{2}\right)d^{n-i-2}\left(d-a\right)^2
+\left(d-a\right)^{n-i}
\end{equation}
\end{theorem}
\begin{proof}
See the proof of Theorem $7.$
\end{proof}
\begin{theorem}
For every given $1\ne k\ne m\in \mathbb{Z^{+}}$ and every given complex number $\gamma=\gamma_1+\gamma_2i$ then $a_1,a_2\in \mathbb{R}$,  $d_1,d_2\in \mathbb{R}|_{\{0\}}$ can be obtained from
\begin{equation}
d_1=\frac{\cos(\gamma_2\ln(k))k^{\gamma_1}-\cos(\gamma_2\ln(m))m^{\gamma_1}}{k-m}
\end{equation}
\begin{equation}
a_1=\cos(\gamma_2\ln(k))k^{\gamma_1}-(k-1)d_1
\end{equation}
\begin{equation}
d_2=\frac{\sin(\gamma_2\ln(k))k^{\gamma_1}-\sin(\gamma_2\ln(m))m^{\gamma_1}}{k-m}
\end{equation}
\begin{equation}
a_2=\sin(\gamma_2\ln(k))k^{\gamma_1}-(k-1)d_2
\end{equation}
\end{theorem}
\begin{proof}Suppose for $1\ne k\ne m\in \mathbb{Z^{+}}$,
\begin{equation*}
A^1(k,a_1,a_2,d_1,d_2)+iB^1(k,a_1,a_2,d_1,d_2)=A^1(m,a_1,a_2,d_1,d_2)+iB^1(m,a_1,a_2,d_1,d_2)
\end{equation*}
\begin{equation*}
\Rightarrow A^1(k,a_1,a_2,d_1,d_2)+iB^1(k,a_1,a_2,d_1,d_2)=\gamma=A^1(m,a_1,a_2,d_1,d_2)+iB^1(m,a_1,a_2,d_1,d_2)
\end{equation*}
\begin{equation*}
\Rightarrow A^1(k,a_1,a_2,d_1,d_2)+iB^1(k,a_1,a_2,d_1,d_2)=\gamma=\gamma_1+\gamma_2i
\end{equation*}
\begin{equation*}
\Rightarrow A^1(m,a_1,a_2,d_1,d_2)+iB^1(m,a_1,a_2,d_1,d_2)=\gamma=\gamma_1+\gamma_2i
\end{equation*}
Then we have the relation:
\begin{equation*}
A^1(k,a_1,a_2,d_1,d_2)=\gamma_1
\end{equation*}
\begin{equation*}
B^1(k,a_1,a_2,d_1,d_2)=\gamma_2
\end{equation*}
\begin{equation*}
A^1(m,a_1,a_2,d_1,d_2)=\gamma_1
\end{equation*}
\begin{equation*}
B^1(m,a_1,a_2,d_1,d_2)=\gamma_2
\end{equation*}
Now from these relations, we have:
\begin{equation*}
B^1(k,a_1,a_2,d_1,d_2)=\gamma_2\Rightarrow \frac{1}{\ln(k)}\arctan\bigg(\frac{a_2+(k-1)d_2}{a_1+(k-1)d_1}\bigg)=\gamma_2
\end{equation*}

\begin{equation}
\Rightarrow \frac{a_2+(k-1)d_2}{a_1+(k-1)d_1}=\tan(\gamma_2\ln(k))
\end{equation}
\begin{equation*}
B^1(m,a_1,a_2,d_1,d_2)=\gamma_2\Rightarrow \frac{1}{\ln(m)}\arctan\bigg(\frac{a_2+(m-1)d_2}{a_1+(m-1)d_1}\bigg)=\gamma_2
\end{equation*}
\begin{equation}
\Rightarrow \frac{a_2+(m-1)d_2}{a_1+(m-1)d_1}=\tan(\gamma_2\ln(m))
\end{equation}

\begin{equation*}
A^1(k,a_1,a_2,d_1,d_2)=\gamma_1\Rightarrow \frac{1}{\ln(k)}\ln\bigg(\frac{a_1+(k-1)d_1}{\cos(B^{1}\ln(k))}\bigg)=\gamma_1
\end{equation*}
\begin{equation}
\Rightarrow a_1+(k-1)d_1=\cos(\gamma_2\ln(k))k^{\gamma_1}
\end{equation}

\begin{equation*}
A^1(m,a_1,a_2,d_1,d_2)=\gamma_1\Rightarrow \frac{1}{\ln(k)}\ln\bigg(\frac{a_1+(m-1)d_1}{\cos(B^{1}\ln(m))}\bigg)=\gamma_1
\end{equation*}

\begin{equation}
\Rightarrow a_1+(m-1)d_1=\cos(\gamma_2\ln(m))m^{\gamma_1}
\end{equation}
From equation $\color{blue}(22)$ and $\color{blue}(24)$, we have

\begin{equation*}
 \frac{a_2+(k-1)d_2}{\cos(\gamma_2\ln(k))k^{\gamma_1}}=\tan(\gamma_2\ln(k))
\end{equation*}
\begin{equation*}
\Rightarrow a_2+(k-1)d_2=\tan(\gamma_2\ln(k))\cos(\gamma_2\ln(k))k^{\gamma_1}=\sin(\gamma_2\ln(k))k^{\gamma_1}
\end{equation*}

\begin{equation}
\Rightarrow a_2+(k-1)d_2=\sin(\gamma_2\ln(k))k^{\gamma_1}
\end{equation}
similarly 
\begin{equation}
a_2+(m-1)d_2=\sin(\gamma_2\ln(m))m^{\gamma_1}
\end{equation}
From equations $\color{blue}(26)$ and $\color{blue}(27)$, we get
\begin{equation}
d_2=\frac{\sin(\gamma_2\ln(k))k^{\gamma_1}-\sin(\gamma_2\ln(m))m^{\gamma_1}}{k-m}
\end{equation}
And then from equation $\color{blue}(26)$ or $\color{blue}(27)$, we get
\begin{equation}
a_2=\sin(\gamma_2\ln(k))k^{\gamma_1}-(k-1)d_2
\end{equation}

From equation $\color{blue}(24)$ and equation $\color{blue}(25)$, we have
\begin{equation}
d_1=\frac{\cos(\gamma_2\ln(k))k^{\gamma_1}-\cos(\gamma_2\ln(m))m^{\gamma_1}}{k-m}
\end{equation}
and then from equation $\color{blue}(24)$ or $\color{blue}(25)$, we have
\begin{equation}
a_1=\cos(\gamma_2\ln(k))k^{\gamma_1}-(k-1)d_1
\end{equation}
\end{proof}
\begin{theorem}
{\bf (Closed-form formula of $\zeta(Z)$ with real part of $Z$ greater than one)}\\
Define the given complex number, $Z$ with real part of $Z$ greater than one, for real numbers $x_1,x_2$ and $n=3,4,5\cdots$ by:
\begin{equation*}
Z=(n-1)x_1+(n-1)x_2i, \text{ where }i=\sqrt{-1}
\end{equation*}that is $\gamma_1=(n-1)x_1,\gamma_2=(n-2)x_2$, then
for every given $1\ne k\ne m\in \mathbb{Z^{+}}$ and for all $a_1,a_2\in \mathbb{R}$,  $d_1,d_2\in \mathbb{R}|_{\{0\}}$ such that   
\begin{equation*}
d_1=\frac{\cos(\gamma_2\ln(k))k^{\gamma_1}-\cos(\gamma_2\ln(m))m^{\gamma_1}}{k-m}
\end{equation*}And
\begin{equation*}
a_1=\cos(\gamma_2\ln(k))k^{\gamma_1}-(k-1)d_1
\end{equation*}
\begin{equation*}
d_2=\frac{\sin(\gamma_2\ln(k))k^{\gamma_1}-\sin(\gamma_2\ln(m))m^{\gamma_1}}{k-m}
\end{equation*}
And therefore
\begin{equation*}
a_2=\sin(\gamma_2\ln(k))k^{\gamma_1}-(k-1)d_2
\end{equation*}then  a closed form of Riemann zeta function, $\zeta(Z)$ is given by
\begin{equation*}
\zeta(Z)=\sum_{r=1}^{\infty}\frac{1}{r^Z}
\end{equation*}

\begin{equation}
=\frac{(n-1)!(n-3)!}{d}\sum_{i=0}^{n-3}\frac{1}{i!(n-i)!(n-3-i)!} \bigg(\frac{d}{2}\bigg)^i(-1)^iS_{n-i}
\end{equation}
Where 
\begin{equation*}
S_{n-i}
\end{equation*}
\begin{equation}
=\left(\frac{n-i}{2}-1\right)\left(d-a\right)d^{n-i-1}-\frac{n-i}{2}d^{n-i-2}\left(d^2-a^2\right)+d^{n-i}-a^{n-i}
\end{equation}

\end{theorem}
\begin{proof}
Follows from Theorem$(9)$, by letting $\gamma=(n-1)Z$.
\end{proof}
\begin{theorem}
{\bf (Closed-form formula of $\eta(Z)$ with real part of $Z$ greater than zero)}\\
Define the given complex number, $Z$ with real part of $Z$ greater than zero, for real numbers $x_1,x_2$ and $n=3,4,5\cdots$ by:
\begin{equation*}
Z=(n-1)x_1+(n-1)x_2i, \text{ where }i=\sqrt{-1}
\end{equation*}that is $\gamma_1=(n-1)x_1,\gamma_2=(n-2)x_2$, then
for every given $1\ne k\ne m\in \mathbb{Z^{+}}$ and for all $a_1,a_2\in \mathbb{R}$,  $d_1,d_2\in \mathbb{R}|_{\{0\}}$ such that   
\begin{equation*}
d_1=\frac{\cos(\gamma_2\ln(k))k^{\gamma_1}-\cos(\gamma_2\ln(m))m^{\gamma_1}}{k-m}
\end{equation*}And
\begin{equation*}
a_1=\cos(\gamma_2\ln(k))k^{\gamma_1}-(k-1)d_1
\end{equation*}
\begin{equation*}
d_2=\frac{\sin(\gamma_2\ln(k))k^{\gamma_1}-\sin(\gamma_2\ln(m))m^{\gamma_1}}{k-m}
\end{equation*}
And therefore
\begin{equation*}
a_2=\sin(\gamma_2\ln(k))k^{\gamma_1}-(k-1)d_2
\end{equation*}then a closed form of eta function, $\eta(Z)$ is given by
\begin{equation*}
\eta(Z)=\sum_{r=1}^{\infty}(-1)^{r-1}\frac{1}{r^Z}
\end{equation*}

\begin{equation*}
=\frac{(n-1)!(n-3)!}{d}\sum_{i=0}^{n-3}\frac{1}{i!(n-i)!(n-3-i)!} \bigg(\frac{d}{2}\bigg)^i(-1)^iL_{n-i}
\end{equation*}
Where 
\begin{equation*}
L_{n-i}
\end{equation*}
\begin{equation*}
=\bigg(\frac{n-i}{2}-1\bigg)(d-a)d^{n-i-1}-\left(\frac{n-i}{2}\right)d^{n-i-2}\left(d-a\right)^2
+\left(d-a\right)^{n-i}
\end{equation*}
\end{theorem}
\begin{proof}
See the proof of Theorem-$(10)$.
\end{proof}
\begin{theorem}{\bf(Closed-form formula of $\zeta(s)$ with real part of $s$ greater than zero)}
Closed-form formula of Riemann zeta function, $\zeta(s)$ with real part of a complex number $s$ greater than zero can be obtained from the relation
\begin{equation*}
\zeta(s)=\frac{1}{1-2^{(1-s)}}\eta(s)
\end{equation*}
This relation tells us that the left hand equation, $\zeta(s)$ can be obtained from the right hand side equation, $\eta(s)$ with real part of $s$ greater than zero.
\end{theorem}
\begin{proof}
See the proof of Theorem-10 and Theorem-11.
\end{proof}
\begin{theorem}{\bf(Closed-form formula of $\zeta(s)$ with real part of $s$ less than zero)}
Closed-form formula of Riemann zeta function, $\zeta(s)$ with real part of a complex number $s$ less than zero can be obtained from the relation
\begin{equation*}
\zeta(s)=\Gamma(1-s)(2\pi)^{s-1}2\sin\left( \frac{\pi s}{2}\right)\zeta(1-s)
\end{equation*}
This relation tells us that from the right hand side equation, $\zeta(1-s)$ can be obtained from the left hand side equation, $\zeta(s)$ with real part of $s$ greater than one.
\end{theorem}
\begin{proof}
See the proof of Theorem-10. 
\end{proof}
\begin{theorem}
If $a=d$, then $S_{n-i}=0$, which is defined in Theorem $7$, equation $\color{blue}(15)$. Hence
\begin{equation*}
\sum_{r=1}^{\infty}(r^{-s})^{n-1}=0
\end{equation*}
Where the complex number $s=\gamma_1+i\gamma_2$ with $\gamma_1>1$. Define $n=2,3,\cdots$. For every given $1\ne k\ne m\in \mathbb{Z^{+}}$ and for all $t\in \mathbb{Z}$,
\begin{equation*}
\gamma_2=\frac{t\pi}{\ln(k)-\ln(m)}
\end{equation*}and
\begin{equation*}
\gamma_1=1+\left[ \frac{1}{\ln(k)-\ln(m)}\right]\left[\ln\left( \frac{\cos(\gamma_2\ln(m))}{\cos(\gamma_2\ln(k))}\right)\right]
\end{equation*}
\end{theorem}

\begin{proof}
\begin{equation*}
a=d\Longrightarrow a_1=d_1,a_2=d_2
\end{equation*}If
\begin{equation*}
a_1=d_1
\end{equation*}
Then from equation $\color{blue}(30)$ and equation $\color{blue}(31)$, we have
\begin{equation*}
d_1=\frac{\cos(\gamma_2\ln(k))k^{\gamma_1}-\cos(\gamma_2\ln(m))m^{\gamma_1}}{k-m}
\end{equation*}And
\begin{equation*}
a_1=\cos(\gamma_2\ln(k))k^{\gamma_1}-(k-1)d_1
\end{equation*}
\begin{equation*}
\Rightarrow a_1=\frac{\cos(\gamma_2\ln(k))k^{\gamma_1}}{k}
\end{equation*}
Equate this with equation $\color{blue}(30)$, that is
\begin{equation*}
\frac{\cos(\gamma_2\ln(k))k^{\gamma_1}}{k}=\frac{\cos(\gamma_2\ln(k))k^{\gamma_1}-\cos(\gamma_2\ln(m))m^{\gamma_1}}{k-m}
\end{equation*}
\begin{equation}
\Rightarrow \frac{m^{1-\gamma_1}}{k^{1-\gamma_1}}=\frac{\cos(\gamma_2\ln(m))}{\cos(\gamma_2\ln(k))}
\end{equation}
Similarly from equation $\color{blue}(28)$ and equation $\color{blue}(29)$, we have
\begin{equation}
\frac{m^{1-\gamma_1}}{k^{1-\gamma_1}}=\frac{\sin(\gamma_2\ln(m))}{\sin(\gamma_2\ln(k))}
\end{equation}
Now equate equation $\color{blue}(34)$ with equation $\color{blue}(35)$, then we have
\begin{equation*}
\frac{\cos(\gamma_2\ln(m))}{\cos(\gamma_2\ln(k))}=\frac{\sin(\gamma_2\ln(m))}{\sin(\gamma_2\ln(k))}
\end{equation*}

\begin{equation*}
\Rightarrow \sin(\gamma_2\ln(k)-\gamma_2\ln(m))=0
\end{equation*}Then for all $n\in \mathbb{Z}$, we have
\begin{equation*}
\Rightarrow \gamma_2\ln(k)-\gamma_2\ln(m)=n\pi
\end{equation*}
\begin{equation}
\Rightarrow \gamma_2=\frac{n\pi}{\ln(k)-\ln(m)}
\end{equation}
Now from equation $\color{blue}(34)$ or equation $\color{blue}(35)$, we can find $\gamma_1$, for instance from equation $\color{blue}(34)$, we have
\begin{equation*}
\frac{m^{1-\gamma_1}}{k^{1-\gamma_1}}=\frac{\cos(\gamma_2\ln(m))}{\cos(\gamma_2\ln(k))}
\end{equation*}

\begin{equation*}
\Rightarrow (1-\gamma_1)\ln\left(\frac{m}{k}\right)=\ln\left( \frac{\cos(\gamma_2\ln(m))}{\cos(\gamma_2\ln(k))}\right)
\end{equation*}
\begin{equation}
\Rightarrow \gamma_1=1+\left[ \frac{1}{\ln(k)-\ln(m)}\right]\left[\ln\left( \frac{\cos(\gamma_2\ln(m))}{\cos(\gamma_2\ln(k))}\right)\right]
\end{equation}

\end{proof}
\begin{theorem}
If $a=d$, then $L_{n-i}=0$, which is defined in Theorem-$8$, equation $\color{blue}(17)$. Hence
\begin{equation*}
\sum_{r=1}^{\infty}(-1)^{r-1}(r^{-s})^{n-1}=0
\end{equation*}
Where the complex number $s=\gamma_1+i\gamma_2$ with $\gamma_1>0$. Define $n=2,3,\cdots$. For every given $1\ne k\ne m\in \mathbb{Z^{+}}$ and for all $t\in \mathbb{Z}$,
\begin{equation*}
\gamma_2=\frac{t\pi}{\ln(k)-\ln(m)}
\end{equation*}and
\begin{equation*}
\gamma_1=1+\left[ \frac{1}{\ln(k)-\ln(m)}\right]\left[\ln\left( \frac{\cos(\gamma_2\ln(m))}{\cos(\gamma_2\ln(k))}\right)\right]
\end{equation*}
\end{theorem}
\begin{proof}
See the proof of Theorem-$(14)$.
\end{proof}
\begin{corollary}
$\zeta(s)$ and $\eta(s)$ have the same zeros.
\end{corollary}
\begin{proof}
The consequence of Theorem-12, Theorem-14 and Theorem-15.
\end{proof}
\begin{corollary}
{\bf(Riemann Hypothesis)}
\begin{equation*}
\zeta(s)=\sum_{r=1}^{\infty}r^{-s}=0
\end{equation*}
Where $s=\frac{1}{2}+i\gamma_2$. For all $t\in \mathbb{Z}$ and for all $1\ne k\ne m\in \mathbb{Z^{+}}$ suchthat
\begin{equation*}
-\sqrt{\frac{k}{m}}\le \cos\left(\frac{t\pi\ln(k)}{\ln(k)-\ln(m)}\right)\le\sqrt{\frac{k}{m}}\text{ and }-\sqrt{\frac{k}{m}}\le \sin\left(\frac{t\pi\ln(k)}{\ln(k)-\ln(m)}\right)\le\sqrt{\frac{k}{m}}
\end{equation*}
 we have
\begin{equation*}
\gamma_2=\frac{1}{\ln(m)}\left[\arccos\left(\sqrt{\frac{m}{k}}\cos\left(\frac{t\pi\ln(k)}{\ln(k)-\ln(m)}\right)\right)\right]
\end{equation*}Or
\begin{equation*}
\gamma_2=\frac{1}{\ln(m)}\left[\arcsin\left(\sqrt{\frac{m}{k}}\sin\left(\frac{t\pi\ln(k)}{\ln(k)-\ln(m)}\right)\right)\right]
\end{equation*}
\end{corollary}
\begin{proof}From Theorem $(15)$, take $n=2$ and $\gamma_1=\frac{1}{2}$. Hence from equation $\color{blue}(34)$ and $\color{blue}(35)$:
\begin{equation*}
\sqrt{\frac{m}{k}}=\frac{\cos(\gamma_2\ln(m))}{\cos(\gamma_2\ln(k))}=\frac{\sin(\gamma_2\ln(m))}{\sin(\gamma_2\ln(k))}
\end{equation*}
From \begin{equation*}
\frac{\cos(\gamma_2\ln(m))}{\cos(\gamma_2\ln(k))}=\frac{\sin(\gamma_2\ln(m))}{\sin(\gamma_2\ln(k))}
\end{equation*}We get, for all $t\in \mathbb{Z}$
\begin{equation*}
\gamma_2=\frac{t\pi}{\ln(k)-\ln(m)}
\end{equation*}
From 
\begin{equation*}
\sqrt{\frac{m}{k}}=\frac{\cos(\gamma_2\ln(m))}{\cos(\gamma_2\ln(k))}
\end{equation*}
\begin{equation*}
\Rightarrow \sqrt{\frac{m}{k}}\cos(\gamma_2\ln(k))=\cos(\gamma_2\ln(m))
\end{equation*}
\begin{equation*}
\Rightarrow \arccos\left(\sqrt{\frac{m}{k}}\cos(\gamma_2\ln(k))\right)=\gamma_2\ln(m)
\end{equation*}
Hence for all $t\in \mathbb{Z}$ and for all $1\ne k\ne m\in \mathbb{Z^{+}}$, we have
\begin{equation}
\gamma_2=\frac{1}{\ln(m)}\left[\arccos\left(\sqrt{\frac{m}{k}}\cos\left(\frac{t\pi}{\ln(k)-\ln(m)}\ln(k)\right)\right)\right]
\end{equation}
Similarly, from
\begin{equation*}
\sqrt{\frac{m}{k}}=\frac{\sin(\gamma_2\ln(m))}{\sin(\gamma_2\ln(k))}
\end{equation*}
\begin{equation*}
\Rightarrow \sqrt{\frac{m}{k}}\sin(\gamma_2\ln(k))=\sin(\gamma_2\ln(m))
\end{equation*}
\begin{equation*}
\Rightarrow \arcsin\left(\sqrt{\frac{m}{k}}\sin(\gamma_2\ln(k))\right)=\gamma_2\ln(m)
\end{equation*}
Hence for all $t\in \mathbb{Z}$ and for all $1\ne k\ne m\in \mathbb{Z^{+}}$, we have
\begin{equation}
\gamma_2=\frac{1}{\ln(m)}\left[\arcsin\left(\sqrt{\frac{m}{k}}\sin\left(\frac{t\pi}{\ln(k)-\ln(m)}\ln(k)\right)\right)\right]
\end{equation}

\end{proof}
\section{Conclusion}
\label{S4}
\begin{conclusion}
Riemann Zeta function has no unique solution for the same value. That is, because there is the same $s=\gamma_1+i\gamma_2$ suchthat $\zeta(s)=0$ and $\zeta(s)\ne 0$. It is also true that for the same value of $s=\gamma_1+i\gamma_2$, we do have different values of $\zeta(s)$. We can check this from {\bf Theorem-$10$} and {\bf Theorem-$14$}. For the same value of $\gamma_1$ and $\gamma_2$ taken from {\bf Theorem-$14$}, then we can get infinitely many values of $a_1,a_2,d_1$ $\&$ $d_2$ suchthat $a_1\ne d_1$ and $a_2\ne d_2$ which makes $\zeta(s)\ne 0$ from {\bf Theorem-$10$}.
\end{conclusion}
\begin{conclusion}
$s=\frac{1}{2}+i\gamma_2$ and $s=-2,-4,-6,\cdots$ are not the only zeros of Riemann zeta function. It can be easily understandable that we do have infinitely many zeros of Riemann zeta function from {\bf Theorem-$14$}. 
\end{conclusion}
\begin{conclusion}
Conclusion on Eta function is the same as Riemann zeta function.
\end{conclusion}
\newpage
\bibliographystyle{model1-num-names}
%\bibliographystyle{plain}
%\bibliography{Refs}{}
%References without bibTeX database:

\end{document}